\documentclass[12pt,a4paper]{article}%
\usepackage{amssymb}
\usepackage{amsfonts}
\usepackage{amsmath}
\usepackage{graphicx}%
\providecommand{\U}[1]{\protect\rule{.1in}{.1in}}
\newtheorem{theorem}{Theorem}

\newtheorem{lemma}[theorem]{Lemma}

\newtheorem{proposition}[theorem]{Proposition}
\newtheorem{remark}[theorem]{Remark}

\newenvironment{proof}[1][Proof]{\noindent\textbf{#1.} }{\ \rule{0.5em}{0.5em}}
\begin{document}

\title{A note on the topological synchronization of unimodal maps}
\author{Michele Gianfelice\\Dipartimento di Matematica e Informatica\\Universit\`{a} della Calabria\\Campus di Arcavacata\\Ponte P. Bucci - cubo 30B\\I-87036 Arcavacata di Rende (CS) Italy\\gianfelice@mat.unical.it }
\date{}
\maketitle

\begin{abstract}
In this note we complete the analysis carried on in \cite{CGSV} about the
topological synchronisation of unimodal maps of the interval coupled in a
master-slave configuration, by answering to the questions raised in that
paper. Namely, we compute the weak limits of the invariant measure of the
coupled system as the coupling strength $k\in\left(  0,1\right)  $ tends to
$0$ and to $1$ and discuss the uniqueness of the invariant measure of its
random dynamical system counterpart, proving that the convergence of the
associated Markov chain to its unique stationary measure is geometric.

\end{abstract}

\bigskip

\begin{description}
\item[AMS\ subject classification:] {\small 37A10, 60J10.}

\item[Keywords and phrases:] {\small coupled dynamical systems, unimodal maps,
master-slave system, Markov chains, Random Dynamical Systems, Topological
synchronisation.}
\end{description}

\bigskip

\section{Introduction}

In \cite{CGSV} it has been studied the asymptotic behaviour of the
trajectories of the components of a dynamical system realized by coupling two
unimodal map of the interval in a master/slave configuration. Namely, denoting
by $I$ the interval $[-1,1],$%
\begin{equation}
I^{2}\ni\left(  x_{n},y_{n}^{\left(  k\right)  }\right)  \longmapsto\left\{
\begin{array}
[c]{l}%
x_{n+1}=T_{1}\left(  x_{n}\right) \\
y_{n+1}^{\left(  k\right)  }=(1-k)T_{2}\left(  y_{n}^{\left(  k\right)
}\right)  +kT_{1}\left(  x_{n}\right)
\end{array}
\right.  \in I^{2}\ , \label{msd}%
\end{equation}
where $k\in\left(  0,1\right)  .$

A system of this type were introduced in the physics literature \cite{Letal}%
\ as a simple example of a dynamical system undergoing to the so called
\emph{topological synchronisation}, which refers to the numerically observed
phenomenon, in asymmetrically coupled dynamical systems, of the slave
component attractor to become very similar to the master's one for
sufficiently large values of the coupling. In particular, to gauge the
emergence of this phenomenon the authors of \cite{Letal} suggested to compare
the spectrum of the generalized dimensions of the empirical measures of the
slave component $\left(  D_{q}\left(  \nu_{n}^{\left(  k\right)  }\right)
,\ q\in\mathbb{R}\right)  ,$ where $\nu_{n}^{\left(  k\right)  }:=\frac{1}%
{n}\sum_{i=0}^{n-1}\delta_{y_{i}^{\left(  k\right)  }}$ and
\begin{equation}
D_{q}\left(  \nu_{n}^{\left(  k\right)  }\right)  :=\left\{
\begin{array}
[c]{cc}%
\frac{1}{q-1}\lim_{r\downarrow0}\frac{\log\int_{\Sigma}\nu_{n}^{\left(
k\right)  }\left(  dx\right)  \left[  \nu_{n}^{\left(  k\right)  }\left(
B_{r}\left(  x\right)  \right)  \right]  ^{q-1}}{\log r} & q\neq1\\
\lim_{r\downarrow0}\frac{\int_{\Sigma}\nu_{n}^{\left(  k\right)  }\left(
dx\right)  \log\nu_{n}^{\left(  k\right)  }\left(  B_{r}\left(  x\right)
\right)  }{\log r} & q=1
\end{array}
\right.  \ ,
\end{equation}
provided the limits exist, with that of the master one $\mu_{n}:=\frac{1}%
{n}\sum_{i=0}^{n-1}\delta_{x_{i}},$ for very large values of $n$ and for
Lebesgue almost all initial conditions $\left(  x_{0},y_{0}\right)  \in
I^{2},$ as the coupling increases. In particular, in the case of a system
given in (\ref{msd}) where the coupled maps belong to the logistic family,
they observed numerically that, for very large values of $n,$ the difference
$\Delta D_{q}^{\left(  k\right)  }:=\left\vert D_{q}\left(  \nu_{n}^{\left(
k\right)  }\right)  -D_{q}\left(  \mu_{n}\right)  \right\vert $ converges to
zero as the coupling parameter $k$ approaches to $1$ first for negative values
of $q.$ This lead them to deduce that, assuming $n$ sufficiently large such
that $\rho_{n}^{\left(  k\right)  }:=\frac{1}{n}\sum_{i=0}^{n-1}%
\delta_{\left(  x_{n},y_{i}^{\left(  k\right)  }\right)  }$ is close to one of
its possible weak limits with larger basin of attraction $\rho^{\left(
k\right)  }$ (the physical measure if it exists), as the coupling increases,
the synchronisation of the trajectories of slave component with those of the
master one would start first in the regions of the configurations space of the
coupled system where the invariant measure $\rho^{\left(  k\right)  }$ has low density.

In \cite{CGSV} it has been shown that the topological synchronization
phenomenon can easily be interpreted as the emergence of an invariant set in a
neighbourhood of the diagonal of $I^{2}$ to which the trajectories of the
system components converge in the strong coupling limit. Moreover, it has been
shown that there is a large set of maps, in the family of those of quadratic
type used in \cite{Letal}, such that the physical measure $\mu$ of $T_{1}$ has
density bounded away from zero and consequently $D_{q}\left(  \mu\right)  $
equal to $1$ for all negative values of $q.$ Therefore, the numerical evidence
that $\Delta D_{q}^{\left(  k\right)  }$ can be equal to zero for such values
of the parameter $q$ even for values of the coupling large but bounded away
from $1$ can be explained by the fact that for these values of $k$ also the
weak limit $\nu^{\left(  k\right)  }$ of $\nu_{n}^{\left(  k\right)  }$ has a
density bounded away from zero and so $D_{q}\left(  \nu^{\left(  k\right)
}\right)  =1$ for $q<0.$

Furthermore, in the case, $T_{1}$ is exponentially mixing w.r.t. its invariant
measure $\mu,$ the asymptotic properties of the evolution of the slave
component of (\ref{msd}) has been compared with those of the Markov chain (MC)
$\left\{  Y_{n}^{\left(  k\right)  }\right\}  _{n\geq0}$ such that
$Y_{n+1}^{\left(  k\right)  }=\left(  1-k\right)  T_{2}\left(  Y_{n}^{\left(
k\right)  }\right)  +k\omega_{n},$ where $\left\{  \omega_{n}\right\}
_{n\geq1}$ is a sequence of i.i.d. random variables sampled according to
$\mu.$ In particular, if $\mu$ is a.c. w.r.t. the Lebesgue measure, which in
the following we will denote by $\lambda,$ it has been shown that, for any
$k\in\left[  0,1\right]  ,$ the stationary measures of $\left\{
Y_{n}^{\left(  k\right)  }\right\}  _{n\geq0}$ are a.c. w.r.t. $\lambda$ and
there are at most finitely many ergodic probability measures. It has also been
shown that any stationary measure weakly converges to $\mu$ as $k$ tends to
$1$ and, as $k$ tends to $0,$ to a probability measure invariant under the
evolution defined by $T_{2}.$

In this note we complete the analysis just presented on the coupled dynamical
system (\ref{msd}) answering to the following questions that were left open
\cite{CGSV}.

First, since $T_{1}$ and $T_{2}$ are continuous maps, for any $k\in\left(
0,1\right)  ,$ we can construct an invariant measure $\rho^{\left(  k\right)
}$ for the skew-product system defined by (\ref{msd}). Considering the random
vector $\left(  \eta_{1},\eta_{2}\right)  $ on $I^{2}$ distributed according
to $\rho^{\left(  k\right)  },$ where $\mu$ is the marginal distribution of
the first component representing the master system, and denoting by
$\rho^{\left(  k\right)  }\left(  \cdot|\eta_{1}\right)  $ the conditional
distribution of the r.v. $\eta_{2},$ representing the slave component, w.r.t.
$\eta_{1},$ a natural question arose; denoting by $\rho^{\left(  k\right)
}\left(  \cdot|x\right)  $ the regular version of $\rho^{\left(  k\right)
}\left(  \cdot|\eta_{1}\right)  ,$

\textbf{Question I}: Does $\left\{  \rho^{(k)}\left(  \cdot|x\right)
,\ k\in\left[  0,1\right]  \right\}  $ converges weakly, when $k\rightarrow1,$
to $\mu$ for $\mu-a.e.\ x?$

Second, assuming that the invariant measure $\mu$ of $T_{1}$ has density $h$
and considering the MC $\left\{  Y_{n}^{\left(  k\right)  }\right\}  _{n\geq
0}$ described above one may wonder about the uniqueness of the stationary
measure. More precisely, for any $k\in\left[  0,1\right]  ,$ denoting by
$p_{k}\left(  y,z\right)  $ the transition probability kernel of the MC and by
$L_{k}$ the associated Kolmogorov forward operator
\begin{equation}
L^{1}\left(  \lambda\right)  \ni g\longmapsto\left(  L_{k}g\right)
(z):=\int_{I}dyg(y)p_{k}(y,z)\in L^{1}\left(  \lambda\right)  \ , \label{dfs2}%
\end{equation}
one can ask:

\textbf{Question II}: Is there a unique solution to $L_{k}g^{\left(  k\right)
}=g^{\left(  k\right)  },$ for any $k\in\left(  0,1\right)  ?$

Since numerical simulations (\cite{CGSV} fig. 4A) show that if $\mu$ is a.c.
w.r.t. the Lebesgue measure then, for sufficiently large values of $k,$ the
weak limit as $n\uparrow\infty$ of the empirical measure $\nu_{n}^{\left(
k\right)  }$ is also a.c. w.r.t. the Lebesgue measure, it is reasonable to ask
if, in the case Question II has a positive answer at least for $k$
sufficiently large, a stronger notion of convergence to $\mu$ than the weak
convergence as $k$ tends to $1$ of $\lim\sup_{n\rightarrow\infty}\nu
_{n}^{\left(  k\right)  }$ proven in \cite{CGSV} Proposition 3.1 is in force.
More precisely,

\textbf{Question III}: Assume that for any $k\in\lbrack0,1]$ the solution of
(\ref{dfs2}) is unique. Denoting by $\Delta_{n}^{(k)}$ the total variation
distance between $\nu_{n}^{(k)}$ and $\bar{\nu}_{n}^{(k)}:=\frac{1}{n}%
\sum_{i=0}^{n-1}\delta_{Y_{i}^{\left(  k\right)  }},$ do we have
\begin{equation}
\limsup_{k\rightarrow1}\limsup_{n\rightarrow\infty}\Delta_{n}^{(k)}=0\ ,
\label{q3}%
\end{equation}
for Lebesgue almost all choices of the initial data $x_{0},y_{0}$ and
$\mathbb{P}$-almost all $\omega?$

In fact, for any $n\geq1,$ whenever the intersection of the sets $\left\{
Y_{0}^{\left(  k\right)  },..,Y_{n-1}^{\left(  k\right)  }\right\}  $ and
$\left\{  y_{0}^{\left(  k\right)  },..,y_{n-1}^{\left(  k\right)  }\right\}
$ is empty $\Delta_{n}^{\left(  k\right)  }=1,\forall k\in\left[  0,1\right]
.$ Therefore, a more correct version of the statement in Question III would be
that where (\ref{q3}) is replaced by
\begin{equation}
\limsup_{k\rightarrow1}\left[  \limsup_{n\rightarrow\infty}\left(  \left\Vert
\nu_{n}^{(k)}-\mu_{k}^{\ast}\right\Vert _{TV}+\left\Vert \bar{\nu}_{n}%
^{(k)}-g^{\left(  k\right)  }\lambda\right\Vert _{TV}\right)  +\left\Vert
\mu_{k}^{\ast}-g^{\left(  k\right)  }\lambda\right\Vert _{TV}\right]  =0\ ,
\label{q3'}%
\end{equation}
where $\left\Vert \cdot\right\Vert _{TV}$ denotes the total variation norm,
$\mu_{k}^{\ast}$ is any weak limit of $\left\{  \nu_{n}^{(k)}\right\}
_{n\geq0}$ and $g^{(k)}$ is the density of the unique stationary measure of
$\left\{  Y_{n}^{\left(  k\right)  }\right\}  _{n\geq0}.$

A positive answer to Question III would help to explain the convergence of the
sequence $\left\{  \tilde{\lambda}_{T_{2}}^{\left(  n\right)  }\left(
k\right)  \right\}  _{n\geq1},$ where, $\forall n\geq1,\tilde{\lambda}_{T_{2}%
}^{\left(  n\right)  }\left(  k\right)  :=\frac{1}{n}\sum_{i=0}^{n-1}%
\log\left\vert T_{2}^{\prime}\left(  y_{i}\left(  k,x_{0},y_{0}\right)
\right)  \right\vert ,$ which numerical simulations for
\begin{equation}
I\ni y\longmapsto T_{2}\left(  y\right)  :=c_{2}\left(  1-2y^{2}\right)  \in
I\;,\;c_{2}\in\left(  0,1\right)  \label{T_2}%
\end{equation}
show to hold in the limit as $k\uparrow1$ (\cite{CGSV} fig. 5B), although the
function $\log\left\vert T_{2}^{\prime}\right\vert $ is not bounded on $I$ yet
integrable w.r.t. the Lebesgue measure. A possible explaination of the
integrability of $\log\left\vert T_{2}^{\prime}\right\vert $ w.r.t. a weak
limit $\mu_{k}^{\ast}$ of $\left\{  \nu_{n}^{(k)}\right\}  _{n\geq0},$ for
values of $k$ sufficiently close to $1$ (see Proposition 3.1 in \cite{CGSV}),
could rely on the fact that $\mu_{k}^{\ast}$ has density and $\log\left\vert
T_{2}^{\prime}\right\vert \frac{d\mu_{k}^{\ast}}{d\lambda}$ is $\lambda
$-integrable. Here we will prove that Question III has a negative answer,
which leaves the convergence of $\left\{  \tilde{\lambda}_{T_{2}}^{\left(
n\right)  }\left(  k\right)  \right\}  _{n\geq1}$ as an open problem.

In the next section we will first answer to the second question in the case
the slave map $T_{2}$ is of the form considered in \cite{Letal} and
\cite{CGSV}, then turn to the answer of Question I and deduce from this the
answer in the negative to the third question.

\section{Results}

\subsection{On the uniqueness of the stationary measure of the MC analogue to
the slave system}

Let us assume that $\mu,$ the invariant measure of the master map $T_{1},$ has
density $h$ whose support is strictly contained in $I.$ Namely $T_{1}$ can be
chosen to satisfy the hypothesis of Proposition 2.7 in \cite{BS} (see
Proposition 5.1 in \cite{CGSV}), and consider the probability space $\left(
\Omega,\mathcal{F},\mathbb{P}\right)  $ such that $\Omega:=I^{\mathbb{N}%
},\mathcal{F}$ is the $\sigma$-algebra generated by the cylinder sets
\begin{equation}
\{\omega\in\Omega:(\omega_{1},..,\omega_{n})\in B\},\ n\in\mathbb{N}%
,\ B\in\mathcal{B}(I^{n})\ ,
\end{equation}
where $\mathcal{B}(I^{n})$ is the Borel $\sigma$-field on $I^{n},$ and
$\mathbb{P}:=\mu^{\otimes\mathbb{N}}.$

In \cite{CGSV}, for any $k\in\left(  0,1\right)  ,$ we considered the Random
Dynamical System (RDS) defined by the skew-product
\begin{equation}
\Omega\times I\ni\left(  \omega,y\right)  \longmapsto\Psi_{k}\left(
\omega,y\right)  :=(S\omega,F_{\pi\omega}^{(k)}(y))\in\Omega\times I\ ,
\label{RDS}%
\end{equation}
where $\pi:\Omega\rightarrow I$ is such that $\pi\omega=\omega_{1}%
,S:\Omega\rightarrow\Omega$ is the left shift, so that $\forall n\geq
1,\omega_{n}=\pi S^{n}\omega,$ and
\begin{equation}
I\ni y\longmapsto F_{\pi\omega}^{(k)}(y):=(1-k)T_{2}(y)+k\omega_{1}\in I\ .
\label{Fk}%
\end{equation}
Denoting by $F_{j}^{\left(  k\right)  }:=F_{\pi S^{j}\omega}^{(k)},j\geq0,$
the sequence $\left\{  Y_{n}^{\left(  k\right)  }\right\}  _{n\geq0}$ such
that $Y_{0}^{\left(  k\right)  }$ is a r.v. and, for any $n\geq0,Y_{n+1}%
^{\left(  k\right)  }=F_{n}^{\left(  k\right)  }(Y_{n}^{\left(  k\right)  }),$
is a homogeneous MC (see e.g. \cite{Ki} or \cite{Ar}) whose transition
probability kernel, computed in \cite{CGSV} section 6.2, is
\begin{equation}
p_{k}(y,z)dz=\frac{1}{k}h\left(  \frac{z-(1-k)T_{2}(y)}{k}\right)
\mathbf{1}_{[-1,1]}\left(  \frac{z-(1-k)T_{2}(y)}{k}\right)  dz\ . \label{p_k}%
\end{equation}
As already remarked in \cite{CGSV} (section 6.2, footnote 10), if $T_{1}$
mixes exponentially w.r.t. $\mu,$ the asymptotic statistical properties of the
evolution of the slave component of (\ref{msd}) can be compared with those of
the MC $\left\{  Y_{n}^{\left(  k\right)  }\right\}  _{n\geq0}.$

From Proposition 6.1 in \cite{CGSV} there are at most finitely many stationary
measures of the MC $\left\{  Y_{n}^{\left(  k\right)  }\right\}  _{n\geq0}$
and their densities are fixed points of (\ref{dfs2}). Here we prove that in
fact there is just one, which provide a positive answer to Question II. The
proof rely on the existence of a Lyapunov function for the MC $\left\{
Y_{n}^{\left(  k\right)  }\right\}  _{n\geq0},$ for any $k\in\left(
0,1\right)  ,$ that is a measurable function $V_{k}:\mathbb{R}\longrightarrow
\mathbb{R}_{+}$ such that $\lim_{\left\vert x\right\vert \uparrow+\infty}%
V_{k}\left(  x\right)  =+\infty,$ satisfying
\begin{equation}
\int_{\mathbb{R}}dzp_{k}\left(  y,z\right)  V_{k}\left(  z\right)  \leq
\gamma_{k}V_{k}\left(  y\right)  +K_{k}\ ,\;\gamma_{k}\in\lbrack0,1),K_{k}%
\geq0\ . \label{Lj}%
\end{equation}
We refer the reader to \cite{LM} section 5.7 and to \cite{MT} for the terminology.

\begin{lemma}
\label{LLj}Let $T_{2}:I\circlearrowleft$ be such that $T_{2}\left(  I\right)
\subseteq\left[  -c_{2},c_{2}\right]  \subset I.$ For any $k\in\left(
0,1\right)  $ the function
\begin{equation}
\mathbb{R}\ni x\longmapsto V_{k}\left(  x\right)  =\mathbf{1}_{I}\left(
x\right)  \cosh\left(  c_{2}^{\frac{1}{k}}\left(  1-c_{2}^{2}\right)  \frac
{x}{1-x^{2}}\right)  \in\lbrack1,+\infty)
\end{equation}
is a Lyapunov function for the MC $\left\{  Y_{n}^{\left(  k\right)
}\right\}  _{n\geq0}.$
\end{lemma}

\begin{proof}
Setting $x=\frac{z-(1-k)T_{2}(y)}{k},$ by the convexity of $V_{k},$%
\begin{gather}
\int_{I}dzp_{k}\left(  y,z\right)  \cosh\left(  c_{2}^{\frac{1}{k}}\left(
1-c_{2}^{2}\right)  \frac{z}{1-z^{2}}\right)  =\int_{I}dz\frac{1}{k}h\left(
\frac{z-(1-k)T_{2}(y)}{k}\right)  \times\\
\times\mathbf{1}_{I}\left(  \frac{z-(1-k)T_{2}(y)}{k}\right)  \cosh\left(
c_{2}^{\frac{1}{k}}\left(  1-c_{2}^{2}\right)  \frac{z}{1-z^{2}}\right)
\nonumber\\
=\int_{I}dxh\left(  x\right)  \cosh\left[  c_{2}^{\frac{1}{k}}\left(
1-c_{2}^{2}\right)  \frac{\left(  kx+\left(  1-k\right)  T_{2}\left(
y\right)  \right)  }{1-\left(  kx+\left(  1-k\right)  T_{2}\left(  y\right)
\right)  ^{2}}\right] \nonumber\\
\leq\int_{I}dxh\left(  x\right)  \left\{  k\cosh\left[  c_{2}^{\frac{1}{k}%
}\left(  1-c_{2}^{2}\right)  \left(  \frac{x}{1-x^{2}}\right)  \right]
+\left(  1-k\right)  \cosh\left[  c_{2}^{\frac{1}{k}}\left(  1-c_{2}%
^{2}\right)  \frac{T_{2}\left(  y\right)  }{1-\left(  T_{2}\left(  y\right)
\right)  ^{2}}\right]  \right\} \nonumber\\
\leq\left(  1-k\right)  \cosh\left[  c_{2}^{\frac{1}{k}}\left(  1-c_{2}%
^{2}\right)  \frac{T_{2}\left(  y\right)  }{1-\left(  T_{2}\left(  y\right)
\right)  ^{2}}\right]  +k\int_{I}dxh\left(  x\right)  \cosh\left(
c_{2}^{\frac{1}{k}}\left(  1-c_{2}^{2}\right)  \frac{x}{1-x^{2}}\right)
\nonumber\\
\leq\left(  1-k\right)  \cosh\left(  c_{2}^{\frac{1}{k}+1}\right)
\cosh\left(  c_{2}^{\frac{1}{k}}\left(  1-c_{2}^{2}\right)  \frac{y}{1-y^{2}%
}\right)  +k\int_{I}dxh\left(  x\right)  V_{k}\left(  x\right)  \ ,\nonumber
\end{gather}
where we used that, since $T_{2}\left(  I\right)  \in\left[  -c_{2}%
,c_{2}\right]  ,$%
\begin{align}
\sup_{y\in I}\frac{\cosh\left[  c_{2}^{\frac{1}{k}}\left(  1-c_{2}^{2}\right)
\frac{T_{2}\left(  y\right)  }{1-\left(  T_{2}\left(  y\right)  \right)  ^{2}%
}\right]  }{\cosh\left[  c_{2}^{\frac{1}{k}}\left(  1-c_{2}^{2}\right)
\frac{y}{1-y^{2}}\right]  }  &  \leq\frac{\sup_{y\in I}\cosh\left[
c_{2}^{\frac{1}{k}}\left(  1-c_{2}^{2}\right)  \frac{T_{2}\left(  y\right)
}{1-\left(  T_{2}\left(  y\right)  \right)  ^{2}}\right]  }{\inf_{y\in I}%
\cosh\left[  c_{2}^{\frac{1}{k}}\left(  1-c_{2}^{2}\right)  \frac{y}{1-y^{2}%
}\right]  }\\
&  \leq\cosh\left[  c_{2}^{\frac{1}{k}}\left(  1-c_{2}^{2}\right)  \sup_{y\in
I}\frac{T_{2}\left(  y\right)  }{1-\left(  T_{2}\left(  y\right)  \right)
^{2}}\right] \nonumber\\
&  \leq\cosh\left(  c_{2}^{\frac{1}{k}+1}\right)  \ .\nonumber
\end{align}

Hence, (\ref{Lj}) holds with $\gamma_{k}:=\left(  1-k\right)  \cosh\left(
c_{2}^{\frac{1}{k}+1}\right)  <1$\footnote{Since $c_{2}\in(0,1),$ the function
$(0,1)\ni k\longmapsto\gamma_{k}:=\left(  1-k\right)  \cosh\left(
c_{2}^{\frac{1}{k}+1}\right)  \in\left(  0,1\right)  $ is nonnegative and
nonincreasing. Moreover, $\lim_{k\uparrow1}\gamma_{k}=0$ and $\lim
_{k\downarrow0}\gamma_{k}=1.$} and $K_{k}:=k\int_{I}dxh\left(  x\right)
V_{k}\left(  x\right)  $ which is finite because the support of $h$ is
strictly contained in $I.$
\end{proof}

\begin{proposition}
\label{Unique2}Let $T_{2}:I\circlearrowleft$ be continuous and such that
$T_{2}\left(  I\right)  \subseteq\left[  -c_{2},c_{2}\right]  \subset I.$ For
any $k\in\left(  0,1\right)  $ there exists a unique solution to
$L_{k}g^{\left(  k\right)  }=g^{\left(  k\right)  }$ and, for any probability
density $f,\left\{  L_{k}^{n}f\right\}  _{n\geq0}$ converges geometrically to
$g^{\left(  k\right)  }$ in $L^{1}$ as $n$ tends to infinity.
\end{proposition}

\begin{proof}
The thesis will follow from Theorem 5.6.1 in \cite{LM} once we have shown that
$L_{k}$ is constrictive and that exists a set $A\subseteq I$ with positive
Lebesgue measure such that for any density $f,$ there exists $n_{0}\left(
f\right)  $ such that $\forall n>n_{0}\left(  f\right)  ,\left(  L_{k}%
^{n}f\right)  \left(  x\right)  >0$ for Lebesgue almost all $x\in A.$

Following the proof of Proposition 6.1 in \cite{CGSV} we have that the
spectral radius of $L_{k}$ is equal to $1$ and $L_{k}$ is quasi-compact.
Consequently, its essential spectral radius is strictly smaller than $1$ and,
by the theorem in \cite{Ba}, it is uniformly constrictive and therefore, by
definition, constrictive (see also \cite{yushi} fig. 1). Moreover, denoting by
$\left\{  \mathcal{F}_{m}^{\left(  k\right)  }\right\}  _{m\geq0}$ the natural
filtration generated by the Markov chain $\left\{  Y_{m}^{\left(  k\right)
}\right\}  _{m\geq0},\forall n\geq1,k\in\left(  0,1\right)  ,$ by (\ref{Lj}),
\begin{align}
\mathbb{E}\left.  \left[  V_{k}\left(  Y_{n}^{\left(  k\right)  }\right)
\right\vert Y_{0}^{\left(  k\right)  }=y\right]   &  =\mathbb{E}\left.
\left[  \mathbb{E}\left.  \left[  V_{k}\left(  Y_{n}^{\left(  k\right)
}\right)  \right\vert \mathcal{F}_{n-1}^{\left(  k\right)  }\right]
\right\vert Y_{0}^{\left(  k\right)  }=y\right] \\
&  =\mathbb{E}\left.  \left[  \mathbb{E}\left.  \left[  V_{k}\left(
Y_{n}^{\left(  k\right)  }\right)  \right\vert Y_{n-1}^{\left(  k\right)
}\right]  \right\vert Y_{0}^{\left(  k\right)  }=y\right] \nonumber\\
&  \leq\gamma_{k}\mathbb{E}\left.  \left[  V_{k}\left(  Y_{n-1}^{\left(
k\right)  }\right)  \right\vert Y_{0}^{\left(  k\right)  }=y\right]
+K_{k}\ .\nonumber
\end{align}
Iterating the previous inequality we get
\begin{equation}
\mathbb{E}\left.  \left[  V_{k}\left(  Y_{n}^{\left(  k\right)  }\right)
\right\vert Y_{0}^{\left(  k\right)  }=y\right]  \leq K_{k}\frac{1-\gamma
_{k}^{n+1}}{1-\gamma_{k}}+\gamma_{k}^{n}V_{k}\left(  y\right)  \ .
\label{Lyap}%
\end{equation}
By the Markov inequality, $\forall R>0,$%
\begin{equation}
\mathbb{E}\left.  \left[  V_{k}\left(  Y_{n}^{\left(  k\right)  }\right)
\right\vert Y_{0}^{\left(  k\right)  }=y\right]  \geq R\mathbb{E}\left.
\left[  \mathbf{1}_{\left\{  V_{k}>R\right\}  }\left(  Y_{n}^{\left(
k\right)  }\right)  \right\vert Y_{0}^{\left(  k\right)  }=y\right]  \ .
\label{Markov}%
\end{equation}
On the other hand, defining, for any $n\geq1$ and any bounded Lebesgue
measurable function $\phi,$%
\begin{equation}
\left(  \left(  L_{k}^{\ast}\right)  ^{n}\phi\right)  \left(  y\right)
:=\int_{I}dzp_{k}^{n}\left(  y,z\right)  \phi\left(  z\right)  \ ,
\end{equation}
(\ref{Markov}) reads
\begin{equation}
\left(  \left(  L_{k}^{\ast}\right)  ^{n}V_{k}\right)  \left(  y\right)  \geq
R\left(  \left(  L_{k}^{\ast}\right)  ^{n}\mathbf{1}_{\left\{  V_{k}%
>R\right\}  }\right)  \left(  y\right)  \ .
\end{equation}
Consequently, for any probability density $f$ supported on $I,$
\begin{gather}
\int_{I}dx\left(  L_{k}^{n}f\right)  \mathbf{1}_{\left\{  V_{k}>R\right\}
}\left(  x\right)  =\int_{I}dyf\left(  y\right)  \left(  \left(  L_{k}^{\ast
}\right)  ^{n}\mathbf{1}_{\left\{  V_{k}>R\right\}  }\right)  \left(  y\right)
\\
\leq\frac{1}{R}\int_{I}dyf\left(  y\right)  \left(  \left(  L_{k}^{\ast
}\right)  ^{n}V_{k}\right)  \left(  y\right)  =\frac{1}{R}\int_{I}dx\left(
L_{k}^{n}f\right)  V_{k}\left(  x\right) \nonumber
\end{gather}
and, by (\ref{Lyap}),
\begin{align}
\int_{I}dx\left(  L_{k}^{n}f\right)  \mathbf{1}_{\left\{  V_{k}>R\right\}
}\left(  x\right)   &  \leq\frac{1}{R}\int_{I}dx\left(  L_{k}^{n}f\right)
V_{k}\left(  x\right) \label{Markov1}\\
&  \leq\frac{1}{R}\left(  K_{k}\frac{1-\gamma_{k}^{n+1}}{1-\gamma_{k}}%
+\gamma_{k}^{n}\int_{I}dyf\left(  y\right)  V_{k}\left(  y\right)  \right)
\ .\nonumber
\end{align}
Denoting by $D$ the set of probability density, let $D_{0}:=\left\{  f\in
D:\int_{I}dxf\left(  x\right)  V_{k}\left(  x\right)  <\infty\right\}  .$
Given $g\in D_{0},$ let $n_{0}\geq1$ be such that $\gamma_{k}^{n_{0}}\int
_{I}dyg\left(  y\right)  V_{k}\left(  y\right)  <1.$ Then, by (\ref{Markov1}),
for any $n\geq n_{0},$%
\begin{equation}
\int_{I}dx\left(  L_{k}^{n}g\right)  \mathbf{1}_{\left\{  V_{k}>R\right\}
}\left(  x\right)  \leq\frac{1}{R}\left(  \frac{K_{k}}{1-\gamma_{k}}+1\right)
\end{equation}
and so, for $R>\left(  \frac{K_{k}}{1-\gamma_{k}}+1\right)  ,$%
\begin{equation}
\int_{I}dx\left(  L_{k}^{n}g\right)  \mathbf{1}_{\left\{  V_{k}\leq R\right\}
}\left(  x\right)  \geq1-\frac{1}{R}\left(  \frac{K_{k}}{1-\gamma_{k}%
}+1\right)  >0\ .
\end{equation}

Since $V_{k}\in C\left(  I\right)  ,D_{0}$ is dense in $D;$ therefore, setting
$\epsilon:=1-\frac{1}{R}\left(  \frac{K_{k}}{1-\gamma_{k}}+1\right)  ,$ for
any $f\in D$ there exists $g\in D_{0}$ such that $\left\Vert f-g\right\Vert
_{L^{1}\left(  \lambda\right)  }\leq\frac{\epsilon}{2},$ which implies that,
$\forall n\geq n_{0},$%
\begin{align}
\int_{I}dx\left(  L_{k}^{n}f\right)  \mathbf{1}_{\left\{  V_{k}\leq R\right\}
}\left(  x\right)   &  =\int_{I}dx\left(  L_{k}^{n}g\right)  \mathbf{1}%
_{\left\{  V_{k}\leq R\right\}  }\left(  x\right)  +\int_{I}dx\left(
L_{k}^{n}\left(  f-g\right)  \right)  \mathbf{1}_{\left\{  V_{k}\leq
R\right\}  }\left(  x\right) \\
&  \geq\epsilon-\left\Vert f-g\right\Vert _{L^{1}\left(  dx\right)  }\geq
\frac{\epsilon}{2}\ ,\nonumber
\end{align}
but, because $\left(  L_{k}^{n}f\right)  \geq0,$ there exists $A\subseteq
\left\{  x\in I:V_{k}\left(  x\right)  \leq R\right\}  $ of positive Lebesgue
measure such that $\left(  L_{k}^{n}f\right)  \left(  x\right)  >0$ for
Lebesgue almost all $x\in A.$
\end{proof}

We stress that Proposition \ref{Unique2} does not give back an explicit
estimate on the geometric rate of convergence the Markov chain $\left\{
Y_{n}^{\left(  k\right)  }\right\}  _{n\geq0}$ to its unique stationary
measure, what instead the next result will provide to the price, in some
cases, of proving uniqueness only for values of $k$ smaller than a given one
depending on $T_{1}$ and $T_{2},$ under the assumpion that there exists
$\psi\in B_{1}^{+}:=\left\{  f\in L^{1}\left(  \lambda\right)  :f\geq
0,\left\Vert f\right\Vert _{L^{1}}<1\right\}  $ such that $h>\psi.$

Although the proof of the following result is easily seen to hold in this more
general case, we will present it for the case where $T_{1}$ satisfies the
hypothesis of Proposition 2.7 in \cite{BS}, which, as shown in \cite{CGSV},
can have its own interest.

\begin{proposition}
\label{Unique1}There exists $k_{\ast}:=k_{\ast}\left(  T_{1},T_{2}\right)
\in(0,1]$ such that, for any $k\in(0,k_{\ast}),$ the Markov chain
$\{Y_{n}^{\left(  k\right)  }\}_{n\geq0}$ admits a unique stationary
probability density $g^{\left(  k\right)  }$ such that, for any probability
density $f\in D_{0},$ the sequence $\left\{  L_{k}^{n}f\right\}  _{n\geq0}$
converges geometrically to $g^{\left(  k\right)  }$ in $L^{1}\left(
\lambda\right)  $ at rate $\bar{\alpha}_{k}:=\left[  1-\left(  \alpha_{k}%
-\bar{\alpha}\right)  \right]  \vee\left[  \frac{2+R\frac{\bar{\alpha}}{K_{k}%
}\left(  \gamma_{k}+2\frac{K_{k}}{R}\right)  }{2+R\frac{\bar{\alpha}}{K_{k}}%
}\right]  $ for any $\bar{\alpha}\in\left(  0,\alpha_{k}\right)  $ and
$R>\frac{2K_{k}}{1-\gamma_{k}}$ with $\alpha_{k}:=\alpha_{k}\left(  h\right)
\in\left(  0,1\right)  .$
\end{proposition}

\begin{proof}
The proof follows from Proposition 6.1 in \cite{CGSV} and Theorem 1.2 in
\cite{HM} once we have shown that the transition probability kernel $p_{k}$ of
the Markov chain $\left\{  Y_{n}^{\left(  k\right)  }\right\}  _{n\geq0}$
satisfies the assumptions 1 and 2 given in that paper. Namely:

\begin{enumerate}
\item there exists a Lyapunov function $V_{k}$ for the MC $\left\{
Y_{n}^{\left(  k\right)  }\right\}  _{n\geq0}$ such that (\ref{Lj}) holds;

\item the transition probability kernel of the MC $\left\{  Y_{n}^{\left(
k\right)  }\right\}  _{n\geq0}$ satisfies a Doeblin's type condition on a
subset of $I$ of the form $\left\{  x\in I:V_{k}\left(  x\right)  \leq
R\right\}  =:\mathcal{V}_{k}\left(  R\right)  $ for sufficiently large $R.$
More specifically, there exist a probability measure $\tilde{\nu}_{k}$ on
$\left(  I,\mathcal{B}\left(  I\right)  \right)  $ and $\alpha_{k}\in\left(
0,1\right)  $ such that
\begin{equation}
\inf_{y\in\mathcal{V}_{k}\left(  R\right)  }p_{k}\left(  y,A\right)
\geq\alpha_{k}\tilde{\nu}\left(  A\right)  \;,\;R>\frac{2K_{k}}{1-\gamma_{k}%
}\;,\;A\in\mathcal{B}\left(  I\right)  \ ,
\end{equation}
where the constants $\gamma_{k}$ and $K_{k}$ are those appearing in (\ref{Lj}).
\end{enumerate}

Assumption 1 is satisfied by Lemma \ref{LLj}, therefore we are left with the
proof of Assumption 2.

Let $\psi_{0}$ be the non negative $C^{1}\left(  I\right)  $ function
appearing in the representation of $h$ given in Proposition 2.7 formula (11)
of \cite{BS}. Then, $h\geq\psi_{0}$ and, for any $[a,b)\subseteq I,$%
\begin{align}
\int_{I}dzp_{k}\left(  y,z\right)  \mathbf{1}_{[a,b)}\left(  z\right)   &
=\int_{I}dz\frac{1}{k}h\left(  \frac{z-(1-k)T_{2}(y)}{k}\right)
\mathbf{1}_{I}\left(  \frac{z-(1-k)T_{2}(y)}{k}\right)  \mathbf{1}%
_{[a,b)}\left(  z\right)  \\
&  =\int_{I}dxh\left(  x\right)  \mathbf{1}_{[a,b)}\left(  kx+\left(
1-k\right)  T_{2}\left(  y\right)  \right)  \nonumber\\
&  \geq\int_{I}dx\psi_{0}\left(  x\right)  \mathbf{1}_{[a,b)}\left(
kx+\left(  1-k\right)  T_{2}\left(  y\right)  \right)  \ .\nonumber
\end{align}
Denoting by $\left[  a_{0},b_{0}\right]  $ the support of $\psi_{0}$ it
follows that $\left[  a_{0},b_{0}\right]  \subseteq$\textrm{supp}$h\subset
I$\footnote{As an example of $T_{1}$ one can consider a map of the same form
of $T_{2},$ namely $I\ni x\mapsto T_{1}\left(  x\right)  :=c_{1}\left(
1-2x^{2}\right)  ,$ for a suitable choice of $c_{1}\in\left(  0,1\right)  ,$
which can be proven to satisfy the assumption of Proposition 2.7 in \cite{BS}
(see \cite{CGSV} section 5.1). In this case \textrm{supp}$h\subseteq\left[
T_{1}^{2}\left(  0\right)  ,T_{1}\left(  0\right)  \right]  =\left[
T_{1}\left(  c_{1}\right)  ,c_{1}\right]  \subset I.$}. Moreover, setting
$c_{R}\left(  k\right)  :=T_{2}\left(  y_{R}\left(  k\right)  \right)  \wedge
T_{2}\left(  -y_{R}\left(  k\right)  \right)  $ where, since $V_{k}$ is
symmetric, $y_{R}\left(  k\right)  \in I$ is such that $V_{k}\left(
y_{R}\left(  k\right)  \right)  =V_{k}\left(  -y_{R}\left(  k\right)  \right)
=R,$ since $c_{R}\left(  k\right)  >-c_{2}$ we have
\begin{gather}
\inf_{y\in\mathcal{V}_{k}\left(  R\right)  }\int_{I}dx\psi_{0}\left(
x\right)  \mathbf{1}_{[a,b)}\left(  kx+\left(  1-k\right)  T_{2}\left(
y\right)  \right)  \\
=\inf_{y\in\left[  -y_{R}\left(  k\right)  ,y_{R}\left(  k\right)  \right]
}\int_{I}dx\psi_{0}\left(  x\right)  \mathbf{1}_{[a,b)}\left(  kx+\left(
1-k\right)  T_{2}\left(  y\right)  \right)  \nonumber\\
\geq\inf_{w\in\left[  -c_{2},c_{2}\right]  }\int_{\left[  a_{0},b_{0}\right]
\cap\left[  \frac{a-\left(  1-k\right)  w}{k},\frac{b-\left(  1-k\right)
w}{k}\right]  }dx\psi_{0}\left(  x\right)  \ .\nonumber
\end{gather}
Clearly, there exist no Borel sets $[a,b)\subseteq I$ such that
\begin{equation}
\inf_{w\in\left[  -c_{2},c_{2}\right]  }\int_{\left[  a_{0},b_{0}\right]
\cap\left[  \frac{a-\left(  1-k\right)  w}{k},\frac{b-\left(  1-k\right)
w}{k}\right]  }dx\psi_{0}\left(  x\right)  >0\label{cond}%
\end{equation}
if either
\begin{equation}
\frac{-1+\left(  1-k\right)  c_{2}}{k}>b_{0}\ \Longrightarrow\ -1+\left(
1-k\right)  c_{2}>b_{0}k\ \Longrightarrow\ -\left(  1-c_{2}\right)  >\left(
b_{0}+c_{2}\right)  k\ ,
\end{equation}
that is for
\begin{equation}
\left\{
\begin{array}
[c]{ll}%
k<-\frac{1-c_{2}}{b_{0}+c_{2}} & \text{if }b_{0}>-c_{2}\\
k>\frac{1-c_{2}}{\left\vert b_{0}+c_{2}\right\vert } & \text{if }b_{0}<-c_{2}%
\end{array}
\right.  \ ,
\end{equation}
or
\begin{equation}
\frac{1-\left(  1-k\right)  c_{2}}{k}<a_{0}\ \Longrightarrow\ 1-\left(
1-k\right)  c_{2}<a_{0}k\ \Longrightarrow\ \left(  1-c_{2}\right)  <\left(
a_{0}-c_{2}\right)  k\ ,
\end{equation}
that is for
\begin{equation}
\left\{
\begin{array}
[c]{ll}%
k>\frac{1-c_{2}}{a_{0}-c_{2}} & \text{if }a_{0}>c_{2}\\
k<-\frac{1-c_{2}}{c_{2}-a_{0}} & \text{if }a_{0}<c_{2}%
\end{array}
\right.  \ .
\end{equation}
Hence, there exists $[a,b)\subseteq I$ such that (\ref{cond}) holds for any
$k\in\left(  0,k_{\ast}\right)  $ where
\begin{align}
k_{\ast} &  :=k_{\ast}\left(  T_{1},T_{2}\right)  \\
&  =\left\{
\begin{array}
[c]{ll}%
\frac{1-c_{2}}{a_{0}-c_{2}}\vee1 & \text{if }T_{1}\text{ and }T_{2}\text{ are
such that }a_{0}>c_{2}\\
\frac{1-c_{2}}{\left\vert b_{0}+c_{2}\right\vert }\vee1 & \text{if }%
T_{1}\text{ and }T_{2}\text{ are such that }b_{0}<-c_{2}\\
1 & \text{if }T_{1}\text{ and }T_{2}\text{ are such that }a_{0}<c_{2}%
,b_{0}>-c_{2}%
\end{array}
\right.  \ .\nonumber
\end{align}

Then, for any $[a,b)\subseteq I,$%
\begin{align}
\inf_{y\in\mathcal{V}_{k}\left(  R\right)  }\int_{I}dzp_{k}\left(  y,z\right)
\mathbf{1}_{[a,b)}\left(  z\right)   &  =\inf_{y\in\mathcal{V}_{k}\left(
R\right)  }\int_{\left[  \frac{a-\left(  1-k\right)  T_{2}\left(  y\right)
}{k},\frac{b-\left(  1-k\right)  T_{2}\left(  y\right)  }{k}\right]
}dxh\left(  x\right) \\
&  \geq\left\Vert \psi_{0}\right\Vert _{L^{1}}\inf_{w\in\left[  -c_{2}%
,c_{2}\right]  }\frac{\int_{\left(  \frac{a-\left(  1-k\right)  w}{k}\right)
\vee a_{0}}^{\left(  \frac{b-\left(  1-k\right)  w}{k}\right)  \wedge b_{0}%
}dx\psi_{0}\left(  x\right)  }{\left\Vert \psi_{0}\right\Vert _{L^{1}}%
}\ ,\nonumber
\end{align}
which implies $\alpha_{k}=\left\Vert \psi_{0}\right\Vert _{L^{1}}$ and
$\mathcal{B}\left(  I\right)  \ni A\longmapsto\tilde{\nu}\left(  A\right)
=\inf_{w\in\left[  -c_{2},c_{2}\right]  }\frac{\int_{I}dx\psi_{0}\left(
x\right)  \mathbf{1}_{A}\left(  kx+\left(  1-k\right)  w\right)  }{\left\Vert
\psi_{0}\right\Vert _{L^{1}}}\in\left[  0,1\right]  .$

Let us fix $k\in\left(  0,k_{\ast}\right)  $ and choose $R>\frac{2K_{k}%
}{1-\gamma_{k}}.$ Proceding as in \cite{HM}, let us consider on $D_{0}$ the
family of weighted total variation distances
\begin{equation}
d_{\beta}\left(  f_{1},f_{2}\right)  :=\int_{I}\left[  1+\beta\left(
V_{k}\left(  x\right)  -1\right)  \right]  \left\vert f_{1}\left(  x\right)
-f_{2}\left(  x\right)  \right\vert \left(  dx\right)  \ ,\;f_{1},f_{2}\in
D_{0},\beta>0\ .
\end{equation}
By Theorem 1.3 in \cite{HM}, for any $f\in D_{0},n\geq1$ and any $\bar{\alpha
}\in\left(  0,\alpha_{k}\right)  ,$ we get
\begin{equation}
d_{\frac{\bar{\alpha}}{K_{k}}}\left(  L_{k}^{n}f,g^{\left(  k\right)
}\right)  <\left(  \bar{\alpha}_{k}\right)  ^{n}d_{\frac{\bar{\alpha}}{K_{k}}%
}\left(  f,g^{\left(  k\right)  }\right)
\end{equation}
for $\bar{\alpha}_{k}=\left[  1-\left(  \alpha_{k}-\bar{\alpha}\right)
\right]  \vee\left[  \frac{2+R\frac{\bar{\alpha}}{K_{k}}\left(  \gamma
_{k}+2\frac{K_{k}}{R}\right)  }{2+R\frac{\bar{\alpha}}{K_{k}}}\right]  .$
Since, $\left(  \beta\wedge1\right)  d_{1}\leq d_{\beta}\leq\left(  \beta
\vee1\right)  d_{1},$%
\begin{equation}
\left\Vert L_{k}^{n}f-g^{\left(  k\right)  }\right\Vert _{L^{1}}\leq
d_{1}\left(  L_{k}^{n}f,g^{\left(  k\right)  }\right)  \leq\left(  \frac
{\bar{\alpha}}{K_{k}}\vee\frac{K_{k}}{\bar{\alpha}}\right)  \left(
\bar{\alpha}_{k}\right)  ^{n}d_{1}\left(  f,g^{\left(  k\right)  }\right)  \ .
\end{equation}

\end{proof}

To sum up what it has been proved so far we can state the following result
which extends those proved in Propositions \ref{Unique1} and \ref{Unique2}.

\begin{proposition}
Given $h\in D,$ let the map $T_{2}:I\circlearrowleft$ be continuous and such
that $T_{2}\left(  I\right)  \subset I.$ Assume that there exists $\psi\in
B_{1}^{+}$ such that $h>\psi.$ Then,

\begin{enumerate}
\item The MC $\left\{  Y_{n}^{\left(  k\right)  }\right\}  _{n\geq0}$ defined
by the transition probability kernel (\ref{p_k}) has always a unique
stationary measure $\bar{\nu}^{\left(  k\right)  }:=g^{\left(  k\right)
}\lambda$ for all $k$ and for $k\in\left(  0,1\right)  ;$

\item for any $f\in D,$ the convergence of $\left\{  L_{k}^{n}f\right\}
_{n\geq0}$ to $g^{\left(  k\right)  }$ is geometric, i.e. $L_{k}$ has a
spectral gap in $L^{1}$ for all $k\in\left(  0,1\right)  ;$

\item there exists $k_{\ast}:=k_{\ast}\left(  h,T_{2}\right)  \in(0,1]$ such
that, for any $k\in\left(  0,k_{\ast}\right)  $ and any $f\in D_{0},$ one can
find an exlicit bound on the geometric rate of convergence of $\left\{
L_{k}^{n}f\right\}  _{n\geq0}$ to $g^{\left(  k\right)  }.$
\end{enumerate}
\end{proposition}

\begin{proof}
The proofs of the first two statements follow verbatim that of Proposition
\ref{Unique2}, while the proof of the third one is identical to that of
Proposition \ref{Unique1}, in view of the fact that no role is played by the
unimodality of $T_{2}$ in the proofs of Lemma \ref{LLj} and of Proposition 6.1
in \cite{CGSV}.
\end{proof}

\begin{remark}
We point out that in the third statement of the result just given, i.e. in
Proposition \ref{Unique1}, we can drop the requirement of the continuity of
the map $T_{2}$ to the price of proving the convergence of $\left\{
Y_{n}^{\left(  k\right)  }\right\}  _{n\geq0}$ to its unique stationary
measure $\mu_{k}$ in $d_{1}$ only, since in this case, by its definition in
(\ref{Fk}), $I\ni y\longmapsto F_{\pi\omega}^{(k)}(y)\in I$ is not continuous
for $\mu$-almost every $\omega\in\Omega,$ consequently we may not be able to
use the results in \cite{yushi} to prove that $\mu_{k}=g^{\left(  k\right)
}\lambda.$
\end{remark}

\subsection{The weak limits of the invariant measures of the deterministic
system in the very weak and in the very strong coupling regimes}

For any $k\in\left[  0,1\right]  ,$ setting
\begin{equation}
I^{2}\ni\left(  x,y\right)  \longmapsto\psi_{x}^{\left(  k\right)
}(y)=(1-k)T_{2}(y)+kT_{1}(x)\in I\ ,
\end{equation}
the coupled system (\ref{msd}) can be written as a skew-product system on the
product space $I^{2},$ that is
\begin{equation}
I^{2}\ni\left(  x,y\right)  \longmapsto\Theta^{\left(  k\right)
}(x,y)=\left(  T_{1}(x),\psi_{x}^{\left(  k\right)  }(y)\right)  \in I^{2}\ .
\label{skew}%
\end{equation}

From now on we assume that $T_{1}$ and $T_{2}:I\circlearrowleft$ are
continuous and that $T_{1}$ has an invariant measure $\mu.$

These are indeed sufficient conditions, see e.g. \cite{Si}, for the
construction of an invariant measure $\rho^{\left(  k\right)  }$ for
$\Theta^{\left(  k\right)  }$ such that $\mu$ is the marginal of the first
component of the random vector $\left(  \eta_{1}^{\left(  k\right)  },\eta
_{2}^{\left(  k\right)  }\right)  $ with joint probability measure
$\rho^{\left(  k\right)  }$ and $\mathcal{B}\left(  I\right)  \times
I\ni\left(  A,x\right)  \longmapsto\rho^{\left(  k\right)  }\left(
A|x\right)  \in\left[  0,1\right]  $ is a probability kernel such that
\begin{equation}
\int_{I^{2}}\rho^{\left(  k\right)  }\left(  dx,dy\right)  f(x,y)=\int_{I}%
\mu(dx)\int_{I}\rho^{\left(  k\right)  }(dy|x)f(x,y)\ ,
\end{equation}
for any $\rho^{\left(  k\right)  }$-integrable function $f.$

The answer to Question I is the object of the following result.

\begin{proposition}
\label{wl}The random vector $\left(  \eta_{1}^{\left(  k\right)  },\eta
_{2}^{\left(  k\right)  }\right)  $ with joint probability measure
$\rho^{\left(  k\right)  }$ converges in distribution, as $k\uparrow1,$ to the
random vector $\left(  \eta_{1},\eta_{2}\right)  $ such that, for any Borel
set $A\subseteq I^{2},\mathbb{P}\left\{  \left(  \eta_{1},\eta_{2}\right)  \in
A\right\}  =\int_{A}\mu\otimes\lambda\left(  dx,dy\right)  \delta\left(
y-x\right)  .$ In other words, the family of conditional probability
distribution $\left\{  \rho^{(k)}\left(  \cdot|\eta_{1}\right)  ,\ k\in\left[
0,1\right]  \right\}  $ converges weakly, when $k\rightarrow1,$ to the law of
a degenerate random variable constantly equal to $\eta_{1},$ where $\eta_{1}$
has law $\mu.$

Moreover, if $T_{2}$ admits a unique invariant measure $\bar{\nu},$ then, in
the limit as $k\downarrow0,\left\{  \rho^{\left(  k\right)  },\ k\in\left[
0,1\right]  \right\}  $ converges weakly to $\mu\otimes\bar{\nu},$ that is
$\left\{  \rho^{(k)}\left(  \cdot|\eta_{1}\right)  ,\ k\in\left[  0,1\right]
\right\}  $ converges weakly to $\bar{\nu}.$
\end{proposition}

\begin{proof}
We follow the same strategy of Proposition 6.2 in \cite{CGSV}.

Let us consider the characteristic function
\begin{equation}
\mathbb{R}^{2}\ni\left(  t_{1},t_{2}\right)  \longmapsto\varphi_{\left(
\eta_{1}^{\left(  k\right)  },\eta_{2}^{\left(  k\right)  }\right)  }\left(
t_{1},t_{2}\right)  :=\mathbb{E}\left[  e^{i\left\langle \left(  t_{1}%
,t_{2}\right)  ,\left(  \eta_{1}^{\left(  k\right)  },\eta_{2}^{\left(
k\right)  }\right)  \right\rangle }\right]  \in\mathbb{C}%
\end{equation}
of the random vector $\left(  \eta_{1}^{\left(  k\right)  },\eta_{2}^{\left(
k\right)  }\right)  .$ Since $\rho^{\left(  k\right)  }$ is invariant for
$\Theta^{\left(  k\right)  }=\left(  T_{1},\psi_{\cdot}^{\left(  k\right)
}\right)  ,$ i.e. $\Theta_{\ast}^{\left(  k\right)  }\rho^{\left(  k\right)
}=\rho^{\left(  k\right)  },$%
\begin{align}
\varphi_{\left(  \eta_{1}^{\left(  k\right)  },\eta_{2}^{\left(  k\right)
}\right)  }\left(  t_{1},t_{2}\right)   &  =\int_{I^{2}}\rho^{\left(
k\right)  }\left(  dx,dy\right)  \exp\left\{  i\left(  t_{1}x+t_{2}y\right)
\right\} \\
&  =\int_{I^{2}}\left(  \Theta_{\ast}^{\left(  k\right)  }\rho^{\left(
k\right)  }\right)  \left(  dx,dy\right)  \exp\left\{  i\left(  t_{1}%
x+t_{2}y\right)  \right\} \nonumber\\
&  =\int_{I^{2}}\rho^{\left(  k\right)  }\left(  dx,dy\right)  \exp\left\{
i\left(  t_{1}T_{1}\left(  x\right)  +t_{2}\psi_{x}^{\left(  k\right)
}\left(  y\right)  \right)  \right\} \nonumber\\
&  =\int_{I^{2}}\rho^{\left(  k\right)  }\left(  dx,dy\right)  e^{it_{1}%
T_{1}\left(  x\right)  }e^{it_{2}\left[  kT_{1}\left(  x\right)  +\left(
1-k\right)  T_{2}\left(  y\right)  \right]  }\nonumber\\
&  =\int_{I}\mu\left(  dx\right)  e^{i\left(  t_{1}+t_{2}k\right)
T_{1}\left(  x\right)  }\int_{I}\rho^{\left(  k\right)  }\left(  dy|x\right)
e^{it_{2}\left(  1-k\right)  T_{2}\left(  y\right)  }\ .\nonumber
\end{align}
Since $T_{2}\left(  \eta_{2}^{\left(  k\right)  }\right)  \in I,$ for any
$k\in\left[  0,1\right]  ,(1-k)T_{2}\left(  \eta_{2}^{\left(  k\right)
}\right)  \underset{k\uparrow1}{\longrightarrow}0$ a.s.. Hence, by the
conditional dominated convergence theorem,
\begin{equation}
\lim_{k\longrightarrow1}\mathbb{E}\left.  \left[  e^{i\left(  1-k\right)
t_{2}T_{2}\left(  \eta_{2}^{\left(  k\right)  }\right)  }\right\vert \eta
_{1}\right]  =\mathbb{E}\left.  \left[  \lim_{k\longrightarrow1}e^{i\left(
1-k\right)  t_{2}T_{2}\left(  \eta_{2}^{\left(  k\right)  }\right)
}\right\vert \eta_{1}\right]  =1\ ,\;\forall t_{2}\in\mathbb{R}\ . \label{cdc}%
\end{equation}
Furthermore, the family of measures $\{\rho^{(k)},k\in\lbrack0,1]\}$ on
$\left(  I^{2},\mathcal{B}\left(  I^{2}\right)  \right)  $ is tight, hence
sequentially compact; therefore, given a sequence $\{\rho^{(k_{n})}\}_{n\geq
1}\subset\{\rho^{(k)},k\in\lbrack0,1]\},$ where $\{k_{n}\}_{n\geq1}\uparrow1,$
let us consider $\{\rho^{(k_{n_{l}})}\}_{l\geq1},$ with $\{k_{n_{l}}%
\}_{l\geq1}\subset\{k_{n}\}_{n\geq1},$ a weakly convergent subsequence of
$\{\rho^{(k_{n})}\}_{n\geq1}.$ Denoting by $\bar{\rho}$ the weak limit of
$\{\rho^{(k_{n_{l}})}\}_{l\geq1}$ and by $\bar{\rho}\left(  dy|\cdot\right)  $
the conditional probability distribution such that, for any bounded $\bar
{\rho}$-measurable function $f$ on $I^{2},$%
\begin{equation}
\int_{I^{2}}\bar{\rho}\left(  dx,dy\right)  f\left(  x,y\right)  =\int_{I}%
\mu\left(  dx\right)  \int_{I}\bar{\rho}\left(  dy|x\right)  f\left(
x,y\right)  \ ,
\end{equation}
for any $\left(  t_{1},t_{2}\right)  \in\mathbb{R}^{2},$%
\begin{gather}
\int_{I^{2}}\bar{\rho}\left(  dx,dy\right)  \exp\left\{  i\left(  t_{1}%
x+t_{2}y\right)  \right\}  =\int_{I}\mu\left(  dx\right)  \int_{I}\bar{\rho
}\left(  dy|x\right)  e^{i\left(  t_{1}x+t_{2}y\right)  }\\
=\lim_{l\rightarrow\infty}\varphi_{\left(  \eta_{1}^{\left(  k_{n_{l}}\right)
},\eta_{2}^{\left(  k_{n_{l}}\right)  }\right)  }\left(  t_{1},t_{2}\right)
\nonumber\\
=\lim_{l\rightarrow\infty}\mathbb{E}\left[  e^{i\left(  t_{1}+k_{n_{l}}%
t_{2}\right)  T_{1}\left(  \eta_{1}\right)  }\mathbb{E}\left.  \left[
e^{it_{2}\left(  1-k_{n_{l}}\right)  T_{2}\left(  \eta_{2}^{\left(  k_{n_{l}%
}\right)  }\right)  }\right\vert \eta_{1}\right]  \right] \nonumber\\
=\mathbb{E}\left[  \lim_{l\rightarrow\infty}\left(  e^{i\left(  t_{1}%
+k_{n_{l}}t_{2}\right)  T_{1}\left(  \eta_{1}\right)  }\mathbb{E}\left.
\left[  e^{it_{2}\left(  1-k_{n_{l}}\right)  T_{2}\left(  \eta_{2}^{\left(
k_{n_{l}}\right)  }\right)  }\right\vert \eta_{1}\right]  \right)  \right]
\nonumber\\
=\mathbb{E}\left[  \lim_{l\rightarrow\infty}e^{i\left(  t_{1}+k_{n_{l}}%
t_{2}\right)  T_{1}\left(  \eta_{1}\right)  }\lim_{l\rightarrow\infty
}\mathbb{E}\left.  \left[  e^{it_{2}\left(  1-k_{n_{l}}\right)  T_{2}\left(
\eta_{2}^{\left(  k_{n_{l}}\right)  }\right)  }\right\vert \eta_{1}\right]
\right] \nonumber\\
=\mathbb{E}\left[  \lim_{l\rightarrow\infty}e^{i\left(  t_{1}+k_{n_{l}}%
t_{2}\right)  T_{1}\left(  \eta_{1}\right)  }\mathbb{E}\left.  \left[
\lim_{l\rightarrow\infty}e^{it_{2}\left(  1-k_{n_{l}}\right)  T_{2}\left(
\eta_{2}^{\left(  k_{n_{l}}\right)  }\right)  }\right\vert \eta_{1}\right]
\right] \nonumber\\
=\int_{I}\mu\left(  dx\right)  e^{i\left(  t_{1}+t_{2}\right)  T_{1}\left(
x\right)  }=\int_{I}\mu\left(  dx\right)  e^{i\left(  t_{1}+t_{2}\right)
x}\ ,\nonumber
\end{gather}
where the fourth line follows from the third by the Lebesgue's dominated
convergence theorem, the sixth line follows from the fifth by (\ref{cdc}),
while the last line follows by the invariance of $\mu$ w.r.t. the evolution
defined by $T_{1}.$ Hence, $\bar{\rho}\left(  dy|x\right)  =dy\delta\left(
y-x\right)  ,\mu$-a.s. and, since this result holds for any weakly convergent
subsequence $\{\rho^{(k_{n_{l}})}\}_{l\geq1}$ of any sequence $\{\rho
^{(k_{n})}\}_{n\geq1}$ in $\{\rho^{(k)},k\in\lbrack0,1]\},$ we get the thesis.

Proceeding as in the strong coupling limit ($k\rightarrow1$), by the tightness
of $\{\rho^{(k)},k\in\lbrack0,1]\},$ given a sequence $\{\rho^{(k_{n}%
)}\}_{n\geq1}\subset\{\rho^{(k)},k\in\lbrack0,1]\},$ where $\{k_{n}\}_{n\geq
1}\downarrow0,$ let $\{\rho^{(k_{n_{l}})}\}_{l\geq1},$ with $\{k_{n_{l}%
}\}_{l\geq1}\subset\{k_{n}\}_{n\geq1},$ be a weakly convergent subsequence of
$\{\rho^{(k_{n})}\}_{n\geq1}$ converging to $\bar{\rho}.$ Then, denoting by
$\Theta$ the direct-product map
\begin{equation}
I^{2}\ni\left(  x,y\right)  \longmapsto\Theta\left(  x,y\right)  :=\left(
T_{1}\left(  x\right)  ,T_{2}\left(  y\right)  \right)  \in I^{2}\ ,
\end{equation}
$\forall\left(  t_{1},t_{2}\right)  \in\mathbb{R},$ we have
\begin{gather}
\int_{I^{2}}\bar{\rho}\left(  dx,dy\right)  \exp\left\{  i\left(  t_{1}%
x+t_{2}y\right)  \right\}  =\lim_{l\rightarrow\infty}\int_{I^{2}}\rho^{\left(
k_{n_{l}}\right)  }\left(  dx,dy\right)  \exp\left\{  i\left(  t_{1}%
x+t_{2}y\right)  \right\} \\
=\lim_{l\rightarrow\infty}\int_{I^{2}}\left(  \Theta_{\ast}^{\left(  k_{n_{l}%
}\right)  }\rho^{\left(  k_{n_{l}}\right)  }\right)  \left(  dx,dy\right)
\exp\left\{  i\left(  t_{1}x+t_{2}y\right)  \right\} \nonumber\\
=\lim_{l\rightarrow\infty}\int_{I^{2}}\rho^{\left(  k_{n_{l}}\right)  }\left(
dx,dy\right)  e^{i\left(  t_{1}+k_{n_{l}}t_{2}\right)  T_{1}\left(  x\right)
+it_{2}\left(  1-k_{n_{l}}\right)  T_{2}\left(  y\right)  }\nonumber\\
=\lim_{l\rightarrow\infty}\int_{I}\mu\left(  dx\right)  e^{i\left(
t_{1}+k_{n_{l}}t_{2}\right)  T_{1}\left(  x\right)  }\int_{I}\rho^{\left(
k\right)  }\left(  dy|x\right)  e^{it_{2}\left(  1-k_{n_{l}}\right)
T_{2}\left(  y\right)  }\nonumber\\
=\lim_{l\rightarrow\infty}\mathbb{E}\left[  e^{i\left(  t_{1}+k_{n_{l}}%
t_{2}\right)  T_{1}\left(  \eta_{1}\right)  }\mathbb{E}\left.  \left[
e^{it_{2}\left(  1-k_{n_{l}}\right)  T_{2}\left(  \eta_{2}^{\left(  k_{n_{l}%
}\right)  }\right)  }\right\vert \eta_{1}\right]  \right] \nonumber\\
=\int_{I^{2}}\bar{\rho}\left(  dx,dy\right)  \exp\left\{  i\left(  t_{1}%
T_{1}\left(  x\right)  +t_{2}T_{2}\left(  y\right)  \right)  \right\}
\nonumber\\
=\int_{I^{2}}\left(  \Theta_{\ast}\bar{\rho}\right)  \left(  dx,dy\right)
\exp\left\{  i\left(  t_{1}x+t_{2}y\right)  \right\}  \ ,\nonumber
\end{gather}
that is the invariance of $\bar{\rho}$ for the evolution defined by $\Theta.$
In particular, if $\left(  \eta_{1},\eta_{2}\right)  $ is a random vector with
law $\bar{\rho},$ by setting $t_{2}=0$ in the r.h.s. of the last expression we
get that the marginal of $\eta_{1}$ is invariant under the evolution defined
by $T_{1}$ and therefore must be equal to $\mu,$ while, setting $t_{1}=0,$ we
get that $\bar{\nu},$ the marginal of $\eta_{2},$ is invariant under the
evolution defined by $T_{2}.$ Then, if $\bar{\nu}$ is the unique invariant
measure for the evolution defined by $T_{2},\mu\otimes\bar{\nu}$ is the unique
measure left invariant by $\Theta.$
\end{proof}

This result also proves that the answer to Question III is negative.

First notice that if $\mu$ is the unique invariant measure for $T_{1},$
denoting by $\left(  \zeta_{1}^{\left(  n\right)  }\left(  k\right)
,\zeta_{2}^{\left(  n\right)  }\left(  k\right)  \right)  $ the random vector
with law $\rho_{n}^{\left(  k\right)  },\nu_{n}^{\left(  k\right)  }$ is the
maginal w.r.t. $\zeta_{2}^{\left(  n\right)  }\left(  k\right)  .$ Hence,
$\forall k\in\left(  0,1\right)  ,$ any limit in distribution of $\left(
\zeta_{1}^{\left(  n\right)  }\left(  k\right)  ,\zeta_{2}^{\left(  n\right)
}\left(  k\right)  \right)  $ will not be made by a random vector with
independent components even in the limit as $k\uparrow1$ since, as shown in
Proposition \ref{wl}, for any $\varphi\in C\left(  I^{2}\right)
,\lim_{k\rightarrow1}\rho^{\left(  k\right)  }\left(  \varphi\right)
=\int_{I}\mu\left(  dx\right)  \varphi\left(  x,x\right)  .$ On the other
hand, by Proposition \ref{Unique2} and Theorem 2.1.7 in \cite{Ar}, for any
$k\in\left(  0,1\right)  ,\left\{  \bar{\rho}_{n}^{\left(  k\right)
}\right\}  _{n\geq1}$ weakly converges to $\mathbb{P}\otimes\bar{\nu}^{\left(
k\right)  },$ so that $\forall\varphi\in C\left(  I^{2}\right)  ,\left\{
\bar{\rho}_{n}^{\left(  k\right)  }\left(  \varphi\right)  \right\}  _{n\geq
0}$ converges to $\mu\otimes\bar{\nu}^{\left(  k\right)  }\left(
\varphi\right)  ,$ while, as $k$ tends to $1,$ by Proposition 6.2 in
\cite{CGSV}, this measure weakly converges to $\mu\otimes\mu.$

On the other hand, denoting by $f^{\left(  k\right)  }$ be the density of
$\nu_{k}^{\ast}.$ Then,
\begin{align}
\left\Vert g^{\left(  k\right)  }\lambda-f^{\left(  k\right)  }\lambda
\right\Vert _{TV}  &  =\mathbb{E}\left[  \left\vert g^{\left(  k\right)
}-f^{\left(  k\right)  }\right\vert \right]  =\mathbb{E}\left[  \mathbb{E}%
\left[  \left\vert g^{\left(  k\right)  }-f^{\left(  k\right)  }\right\vert
|\eta_{1}\right]  \right] \\
&  \geq\mathbb{E}\left[  \left\vert \mathbb{E}\left[  g^{\left(  k\right)
}-f^{\left(  k\right)  }|\eta_{1}\right]  \right\vert \right]  \ .\nonumber
\end{align}

Since the processes $\left\{  Y_{i}^{\left(  k\right)  }\right\}  _{i\geq0}$
and $\left\{  x_{i}\right\}  _{i\geq0}$ are idependent, $\mathbb{E}\left[
g^{\left(  k\right)  }-f^{\left(  k\right)  }|\eta_{1}\right]  =g^{\left(
k\right)  }-\mathbb{E}\left[  f^{\left(  k\right)  }|\eta_{1}\right]
,\lambda$-a.s., so that
\begin{align}
\left\Vert g^{\left(  k\right)  }\lambda-f^{\left(  k\right)  }\lambda
\right\Vert _{TV}  &  \geq\mathbb{E}\left[  \left\vert g^{\left(  k\right)
}-\mathbb{E}\left[  f^{\left(  k\right)  }|\eta_{1}\right]  \right\vert
\right]  =\left\Vert g^{\left(  k\right)  }-\rho^{\left(  k\right)  }\left(
\cdot|\eta_{1}\right)  \right\Vert _{TV}\\
&  \geq\sup_{\phi\in C\left(  I\right)  \ :\ \left\Vert \phi\right\Vert
_{\infty}\leq1}\left\vert \int_{I}dxg^{\left(  k\right)  }\left(  x\right)
\phi\left(  x\right)  -\int_{I}\rho^{\left(  k\right)  }\left(  dx|\eta
_{1}\right)  \phi\left(  x\right)  \right\vert \ .\nonumber
\end{align}
But, for any $\phi\in C\left(  I\right)  ,$ by from Proposition 6.2 in
\cite{CGSV} and Proposition \ref{wl}
\begin{equation}
\lim\sup_{k\rightarrow1}\left\vert \int_{I}dxg^{\left(  k\right)  }\left(
x\right)  \phi\left(  x\right)  -\int_{I}\rho^{\left(  k\right)  }\left(
dx|\eta_{1}\right)  \phi\left(  x\right)  \right\vert =\left\vert \mu\left(
\phi\right)  -\phi\left(  \eta_{1}\right)  \right\vert \ ,
\end{equation}
which is always positive unless $\phi$ is a constant function.

\section{Conclusions}

In this note we conclude the analysis carried on in \cite{CGSV} about the very
nature of the so called topological synchronisazion phenomenon occurring,
according to the physics literature (\cite{Letal} and referece therein), in
asymmetrically coupled chaotic dynamical systems in the strong coupling
regime, in the paradigmatic case of two logistic type maps coupled in a master
slave configuration.

In particular, we prove that, in the limit as the coupling constant $k$ tends
to zero, any invariant measure of the coupled system $\rho^{\left(  k\right)
}$ weakly converges to the product measure $\mu_{1}\otimes\mu_{2}$ where
$\mu_{1}$ is an invariant measures for the map $T_{1}$ defining the evolution
of the master component and $\mu_{2}$ is an invariant measures for the map
$T_{2}$ appearing in the definition of the evolution of the slave component
given in (\ref{msd}). On the other hand, since $T_{1}$ and $T_{2}$ are both
continuous, the coupled system has an invariant measure $\rho^{\left(
k\right)  }$ admiting the disintegration
\begin{equation}
\int_{I^{2}}\rho^{\left(  k\right)  }\left(  dx,dy\right)  f(x,y)=\int_{I}%
\mu_{1}(dx)\int_{I}\rho^{\left(  k\right)  }(dy|x)f(x,y)\ ,
\end{equation}
where $f$ is a $\rho^{\left(  k\right)  }$-integrable function, and, in the
limit as $k\uparrow1,$ we can prove that the probability kernel $\mathcal{B}%
\left(  I\right)  \times I\ni\left(  A,x\right)  \longmapsto\rho^{\left(
k\right)  }\left(  A|x\right)  \in\left[  0,1\right]  $ weakly converges to a
random measure concentrated on a r.v. distributed according to the law
$\mu_{1}.$ In other words, in the limit $k\uparrow1$ the joint law
$\rho^{\left(  k\right)  }$ of the system's components concentrates on the
diagonal of $I^{2}.$

Moreover, looking at the RDS where the evolution of the slave component of the
coupled system is obtained by substituting in the driving term the evolution
of the master's component with a sequence of i.i.d. random variables sampled
according to the master's component invariant measure $\mu_{1},$ we prove
that, if $T_{2}$ is continuous and $T_{2}\left(  I\right)  \subset I,$ e.g.has
the form (\ref{T_2}), and $\mu_{1}$ has density, the RDS just defined has a
unique invariant measure and the rate of convergence in $L^{1}\left(
\lambda\right)  $ of the associated MC to its unique stationary distribution
is geometric. Furthermore, if $\mu_{1}=h\lambda,$ where $h>\psi$ for some
$\psi\in B_{1}^{+},$ which is the case e.g. if the r.v. with law $\mu_{1}$ is
stochastically dominated by an a.c. r.v., the just mentioned geometric rate of
convergence admits an explicit bound for $k\in\left(  0,k_{\ast}\left(
\mu,T_{2}\right)  \right)  $ with $k_{\ast}\left(  h,T_{2}\right)  \in(0,1].$

\section{Aknowledgements}

MG is partially supported by INDAM-GNAMPA and thanks Sandro Vaienti for useful
discussions. We are also grateful to the referees for the careful reading of
the manuscript and for their comments which helped us to improve the
presentation and the results of the paper.


\begin{thebibliography}{999999999999}


\bibitem[Ar]{Ar}L. Arnold \emph{Random Dynamical Systems}, Springer, (1998).

\bibitem[Ba]{Ba}W. Bartoszek \emph{On uniformly smoothing stochastic
operators} Commentationes Mathematicae Universitatis Carolinae, \textbf{36},
203--206 (1995).

\bibitem[BS]{BS}V. Baladi, D. Smania \emph{Linear response for smooth
deformations of generic nonuniformly hyperbolic unimodal maps}, Annales
scientifiques de l'\'{E}cole Normale Sup\'{e}rieure, \textbf{45}, 4, 861--926 (2012).

\bibitem[BNNT]{yushi}P. G. Barrientos, F. Nakamura, Y. Nakano, H. Toyokawa
\emph{Finitude of physical measures for random maps}, preprint,
https://arxiv.org/pdf/2209.08714.pdf (2022).

\bibitem[CGSV]{CGSV}Caby Th., Gianfelice M., Saussol B., Vaienti S.
\emph{Topological synchronisation or a simple attractor?} Nonlinearity
\textbf{36,} 7, 3603--3621 (2023).

\bibitem[HM]{HM}Hairer M., Mattingly J. \emph{Yet another look at Harris'
ergodic theorem for Markov chains }Seminar on Stochastic Analysis, Random
Fields and Applications VI, Progr. Probab. \textbf{63}, 109--117 (2011).

\bibitem[Ki]{Ki}Yu. Kifer \emph{Ergodic Theory of Random Trasformations},
Springer (1986).

\bibitem[Lahav et al.]{Letal}Lahav N., Sendina-Nadal I., Hens C., Ksherim B.,
Barzel B., Cohen R., Boccaletti S. \emph{Topological synchronization of
chaotic systems} Sci. Rep. \textbf{12}, 2508 (2022).

\bibitem[LM]{LM}Lasota A., Mackey M. C. \emph{Chaos, Fractals and Noise -
Stochastic Aspects of Dynamics }Second Edition, Applied Mathematica Sciences
Vol. 97, Springer (1994).

\bibitem[MT]{MT}Meyn S. P., Tweedie R. L. \emph{Markov chains and stochastic
stability} Second Edition, Cambridge University Press (2009).

\bibitem[Si]{Si}Simmons D \emph{Conditional measures and conditional
expectation; Rohlin's Disintegration Theorem} Discrete and Continuous
Dynamical Systems 32(7): 2565-2582 (2012).
\end{thebibliography}
\end{document}